\documentclass{amsart}%
\usepackage{amsmath}
\usepackage{amsfonts}
\usepackage{amssymb}
\usepackage{tikz-cd}
\usepackage{graphicx}%
\usepackage{color}
\providecommand{\U}[1]{\protect\rule{.1in}{.1in}}
\newtheorem{theorem}{Theorem}[section]
\newtheorem{lemma}[theorem]{Lemma}
\newtheorem{corollary}[theorem]{Corollary}
\newtheorem{proposition}[theorem]{Proposition}
\newtheorem*{theorem*}{Theorem}
\newtheorem*{corollary*}{Corollary}

\theoremstyle{definition}
\newtheorem{definition}[theorem]{Definition}
\newtheorem{example}[theorem]{Example}

\newtheorem{question}[theorem]{Question}

\theoremstyle{remark}
\newtheorem{remark}[theorem]{Remark}

\numberwithin{equation}{section}
\newcommand\numberthis{\addtocounter{equation}{1}\tag{\theequation}}

\begin{document}

\title{Free products with amalgamation over central $\mathrm{C}^*$-subalgebras}
\author{Kristin Courtney}
\address{Mathematical Institute, WWU M\"{u}nster, Einsteinstr. 62, M\"{u}nster}
\email{kcourtne@uni-muenster.de}
\thanks{The research of the first-named author was supported by the Deutsche Forschungsgemeinschaft (SFB 878 Groups, Geometry \& Actions).}

\author{Tatiana Shulman}
\address{Department of Mathematical Physics and Differential Geometry, Institute of Mathematics of Polish Academy of Sciences, Warsaw}
\email{tshulman@impan.pl}
\thanks{The research of the second-named author was supported by the Polish National Science Centre grant
under the contract number DEC- 2012/06/A/ST1/00256, by the grant H2020-MSCA-RISE-2015-691246-QUANTUM DYNAMICS and Polish Government grant 3542/H2020/2016/2, and from the Eric Nordgren
Research Fellowship Fund at the University of New Hampshire.}





\date{}


\commby{Adrian Ioana}

\begin{abstract} Let $A$ and $B$ be $\mathrm{C}^*$-algebras whose quotients are all RFD, 
and let $C$ be a central $\mathrm{C}^*$-subalgebra in both $A$ and $B$. 
We prove that the full amalgamated free product
$A*_C B$ is then RFD. This generalizes Korchagin's result 
that amalgamated free products of commutative $\mathrm{C}^*$-algebras are RFD.
When applied to the case of trivial amalgam, our methods recover the result of Exel and Loring 
for separable $\mathrm{C}^*$-algebras.
As corollaries to our theorem, we give sufficient conditions for amalgamated
free products of maximally almost periodic (MAP) groups to have RFD $\mathrm{C}^*$-algebras and hence to be MAP.
\end{abstract}

\maketitle

\section{Introduction}

We say a $\mathrm{C}^*$-algebra $A$ is {\it residually finite dimensional} (RFD) if it has a separating family of finite dimensional representations. The direct sum of such a family yields a faithful embedding into a direct product of matrix algebras $\prod_{i\in \mathcal{I}} \mathbb{M}_{k_i}$, and so RFD $\mathrm{C}^*$-algebras can be thought of as those which are ``block diagonalizable".
  In addition to this, various other characterizations of the property 
have been obtained over the years (notably \cite{Arc95}, \cite{EL}, \cite{Had14}, \cite{CS}), and numerous classes of $\mathrm{C}^*$-algebras have been shown to be RFD.

Residual finite dimensionality and its permanence properties can be found at the heart of some of the most important questions in operator algebras. 
Perhaps most famously, Kirchberg proved in \cite{Kir93} that Connes' Embedding Problem (\cite{Con76}) is equivalent to the question of whether or not 
$\mathrm{C}^*(\mathbb{F}_2\times \mathbb{F}_2)$ is RFD.
Outside of $\mathbb{F}_2\times \mathbb{F}_2$, new examples of groups whose full group $\mathrm{C}^*$-algebra is RFD have become particularly welcome due to their relevance to problems of finding decidability algorithms for groups (see \cite{FNT}).


In the interest of finding more examples (and non-examples) of RFD $\mathrm{C}^*$-algebras, $\mathrm{C}^*$-algebraists have explored various permanence properties of residual finite dimensionality. In particular, when is it preserved under amalgamated free products?
Given two $\mathrm{C}^*$-algebras $A$ and $B$, each containing a copy of another $\mathrm{C}^*$-algebra $C$, their {\it amalgamated free product} $A*_CB$ is the unique $\mathrm{C}^*$-algebra such that there exist maps $\iota_A:A\to A*_CB$ and $\iota_B:B\to A*_CB$ whose images generate $A*_CB$ and whose restrictions to $C$ agree and such that $A*_C B$ is universal in this regard, meaning any pair of maps $\psi_A:A\to D$ and $\psi_B:B\to D$ into another $\mathrm{C}^*$-algebra $D$ that agree on $C$ must factor through $\iota_A$ and $\iota_B$ respectively. We call $C$ the {\it amalgam}.

The question of when the amalgamated product of two RFD $\mathrm{C}^*$-algebras is again RFD is quite difficult in full generality. For instance, Connes' Embedding Problem can be reformulated as a question of whether a certain amalgamated free product of full group $\mathrm{C}^*$-algebras is RFD. Indeed, for any discrete groups $\Lambda \leq G_1,G_2$, we have that $\mathrm{C}^*(G_1*_\Lambda  G_2)\simeq \mathrm{C}^*(G_1)*_{\mathrm{C}^*(\Lambda )}\mathrm{C}^*(G_2)$ (see \cite[Lemma 3.1]{ESS} for an argument). So we can write
$$\mathrm{C}^*(\mathbb{F}_2\times \mathbb{F}_2)\simeq \mathrm{C}^*(\mathbb{F}_2\times \mathbb{Z})\ast_{\mathrm{C}^*(\mathbb{F}_2)}\mathrm{C}^*(\mathbb{F}_2\times \mathbb{Z}),$$
where
it follows from \cite{Cho80} and the fact that residual finite dimensionality is preserved by minimal tensor products that $\mathrm{C}^*(\mathbb{F}_2\times \mathbb{Z})\simeq \mathrm{C}^*(\mathbb{F}_2)\otimes \mathrm{C}^*(\mathbb{Z})$ is RFD. 

In the case where the amalgam is trivial, i.e., when the amalgamated product is the full free product $A\ast B$ (or unital full free product $A\ast_{\mathbb{C}}B$), Exel and Loring showed in \cite{EL} that the amalgamated product of RFD $\mathrm{C}^*$-algebras is RFD. For non-trivial amalgams, such a nice result is too much to ask. In fact this can fail even in the case of amalgamated products of matrix algebras, as \cite[Example 2.4]{BD} shows. 
Nonetheless, necessary and sufficient conditions have been given for when amalgamated products of two separable RFD $\mathrm{C}^*$-algebras over finite dimensional amalgams are RFD, first for matrix algebras by Brown and Dykema in \cite{BD}, then for finite dimensional $\mathrm{C}^*$-algebras by Armstrong, Dykema, Exel, and Li in \cite{ADEL}, and finally for all separable RFD $\mathrm{C}^*$-algebras by Li and Shen in \cite{LiShen}.
Moving beyond finite dimensional amalgams, group theoretic results and restictions, which we outlilne below, indicate that the next natural class to study is amalgamated products of RFD $\mathrm{C}^*$-algebras over central amalgams.
In \cite{Kor14}, Korchagin proved that any amalgamated product of two commutative $\mathrm{C}^*$-algebras is RFD. This was the first and, until now, the only positive result on amalgamated products of RFD $\mathrm{C}^*$-algebras over infinite dimensional $\mathrm{C}^*$-algebras.

In this paper we substantially generalize Korchagin's statement. Let us say that a $\mathrm{C}^*$-algebra is {\it strongly RFD} if all its quotients are RFD. Here we prove


\begin{theorem*}
Let $A$ and $B$ be separable strongly RFD $\mathrm{C}^*$-algebras and let $C$ be a central $\mathrm{C}^*$-subalgebra in both $A$ and $B$.
Then the amalgamated free product $A\ast_C B$ is RFD.
\end{theorem*}


\noindent In particular, since all commutative $\mathrm{C}^*$-algebras are clearly strongly RFD, this gives the  result of \cite{Kor14} with a different and shorter proof. Moreover, when applied to the case of trivial amalgam, our methods recover the result of Exel-Loring  for separable $\mathrm{C}^*$-algebras. In fact for a non-trivial $C$ and strongly RFD $A$ and $B$, or a trivial $C$ and RFD $A$ and $B$, given an irreducible representation $\rho$ of $A\ast_C B$ and a sequence $p_n \uparrow 1$ of projections on a Hilbert space, we can construct finite dimensional representations of $A\ast_C B$ living on subspaces of $p_nH$  and $\ast$-strongly converging to $\rho$.

As corollaries to our main theorem, we give sufficient conditions for amalgamated free products of discrete and locally compact groups to have RFD $\mathrm{C}^*$-algebras.

 \medskip
 
\noindent {\bf Corollary.} {\it Let $G_1$ and $G_2$ be virtually abelian discrete groups and let $\Lambda $ be a central subgroup in both $G_1$ and $G_2$. Then
the full group $\mathrm{C}^*$-algebra of the amalgamated free product $G_1*_\Lambda  G_2$ is RFD.}

 \medskip

\noindent {\bf Corollary.} {\it Let $G_1$ and $G_2$ be separable locally compact groups, 
and let $\Lambda$ be an open central subgroup in both. Assume moreover that $G_1$ and $G_2$ are Lie groups each containing a closed subgroup of finite index that is compact modulo its center or  are projective limits of such Lie groups. Then the full group $\mathrm{C}^*$-algebra of the amalgamated free product $G_1*_\Lambda G_2$ is RFD.}
\\

 To better understand why amalgamated products over central amalgams are particularly favorable candidates for RFD $\mathrm{C}^*$-algebras, we should pay heed to related results from group theory. The property of being RFD can be considered as a $\mathrm{C}^*$-analogue of maximal almost periodicity and residual finiteness for groups. We say a discrete group is {\it maximally almost periodic} (MAP) if its finite dimensional unitary representations separate its elements, and we say it is {\it residually finite} (RF) if the same can be said for homomorphisms of the group into finite groups. Any discrete RF group is MAP, and by Mal'cev's theorem \cite{Mal40} the converse is true when the group is finitely generated.
 If the full group $\mathrm{C}^*$-algebra $\mathrm{C}^*(G)$ of a discrete group $G$ is RFD, then $G$ is clearly MAP. (In fact, this also holds for locally compact groups, as was shown by Spronk and Wood in \cite{SW};  here the definition of MAP is only changed by the addition of the word ``continuous".) Though the converse does not always hold (e.g. for $\text{SL}_3(\mathbb{Z})$ as shown in \cite{Bek99}), Bekka and Louvet show in \cite{BL00} that when $G$ is amenable, $\mathrm{C}^*(G)$ is RFD exactly when $G$ is MAP. In  particular, this means that a finitely generated amenable group is RF if and only if its full group $\mathrm{C}^*$-algebra is RFD.
This gives us a wealth of examples of RFD $\mathrm{C}^*$-algebras coming from discrete groups.
Beyond these, examples of groups with RFD $\mathrm{C}^*$-algebras include full group $\mathrm{C}^*$-algebras of nonabelian free groups \cite{Cho80},  virtually abelian groups \cite{Thoma68}, surface groups and fundamental groups of closed hyperbolic 3-manifolds that fiber over the circle \cite{LS04}, and many 1-relator groups with non-trivial center \cite{HS}.

Permanence (or lack thereof) of RF and MAP under amalgamation has been well studied in group theory, and results and examples coming from this are valuable guides for the analogous study in $\mathrm{C}^*$-algebras. For instance, in \cite{Hig51}, Higman constructed a pair of finitely generated metabelian groups $G_1$ and $G_2$ with common cyclic subgroup $\Lambda$ such that $G_1*_\Lambda G_2$ is not RF, and Baumslag proved in \cite{Bau63} that any two finitely generated, torsion free, nilpotent groups $G_1$ and $G_2$ that are not abelian have some common subgroup $\Lambda $ such that $G_1*_\Lambda G_2$ is not RF. Since, in both cases, the amalgamated product is still a finitely generated group, we conclude that $\mathrm{C}^*(G_1*_\Lambda G_2)\simeq \mathrm{C}^*(G_1)*_{\mathrm{C}^*(\Lambda)}\mathrm{C}^*(G_2)$ cannot be RFD. These examples give an indication of how restrictive we must be in our choice of algebras and amalgam in the $\mathrm{C}^*$-setting. On the other hand, in \cite{Bau63}, Baumslag proved that the amalgamated product of polycyclic groups over a common central subgroup must be RF, and in \cite{KM} Kahn and Morris proved that the amalgamated product of two topological groups $G_1$ and $G_2$ over a common compact central subgroup $\Lambda$ is MAP if and only if both groups are MAP. These point to central amalgams as the next promising frontier for amalgamated products of RFD $\mathrm{C}^*$-algebras, now that that finite dimensional amalgams are completely understood.

As informative as results for RF and MAP groups are for the study $\mathrm{C}^*$-algebras, it is nice when we can return the favor. 
For $G_1, G_2,$ and $\Lambda $ as in either of our corollaries above, $G_1*_\Lambda G_2$ is MAP. To our knowledge, this fact is new in group theory and gives new examples outside of the result of Kahn and Morris.

\bigskip

 \textbf{Acknowledgements.} We would like to thank Mikhail Ershov and Ben Hayes for useful discussions on amalgamated free products of groups, Don Hadwin for teaching us his unitary orbits result, and Anton Korchagin and Wilhelm Winter for discussions that led to a better exposition of the paper.


 \section{Proofs}

\begin{definition} A $\mathrm{C}^*$-algebra is called strongly RFD if all its quotients are RFD.
\end{definition}

Here are two examples of classes of $\mathrm{C}^*$-algebras that are strongly RFD:


\begin{example}[FDI $\mathrm{C}^*$-algebras]
In \cite{CS}, we call $\mathrm{C}^*$-algebras whose irreducible representations are all finite dimensional FDI.

 Since irreducible representations separate the elements of a $\mathrm{C}^*$-algebra, any FDI $\mathrm{C}^*$-algebra is RFD. Moreover any quotient of an FDI $\mathrm{C}^*$-algebra is again FDI since an irreducible representation of a quotient of a $\mathrm{C}^*$-algebra gives rise to an irreducible representation of the $\mathrm{C}^*$-algebra.

Particular cases of FDI $\mathrm{C}^*$-algebras are subhomogeneous $\mathrm{C}^*$-algebras (i.e., those whose irreducible representations are all of dimension no more than some fixed $n<\infty$)
and continuous fields of finite dimensional $\mathrm{C}^*$-algebras (as defined in \cite{Dix77}).
\end{example}

\begin{example}[RFD just-infinite $\mathrm{C}^*$-algebras]

In \cite{GMR}, Grigorchuk, Musat, and R\o rdam defined a $\mathrm{C}^*$-algebra to be just-infinite when it is infinite dimensional and all of its proper quotients are finite dimensional. In the same paper, they demonstrate the existence of just-infinite RFD $\mathrm{C}^*$-algebras.

Moreover, they prove that there are examples of non-exact, just-infinite RFD $\mathrm{C}^*$-algebras, which means, in particular, that strongly RFD $\mathrm{C}^*$-algebras need not be nuclear.
\end{example}

It is worth noting that no $\mathrm{C}^*$-algebra can be both FDI and just-infinite.  Indeed, by \cite[Lemma 5.4]{GMR},
no RFD just-infinite $\mathrm{C}^*$-algebra is of type I;
while, on the other hand, all FDI algebras are type I.


The following result of Hadwin will be crucial for the proof of the main theorem. It is very close to Theorem 4.3 in \cite{Had77}. The particular formulation below was given in private communication, and so we include a brief proof here.  

\begin{lemma}[Hadwin \cite{Had77}]\label{Don} Let $\{e_n\}$ be an orthonormal basis in a separable Hilbert space $H$, $a, b, c, d \in B(H)$ and  $x = \left(\begin{array}{cc} a&b\\c&d\end{array}\right)$. Then for any unitary $w_n: H \to H\oplus H$ such that $w_ne_k = (e_k, 0)$, $1\le k \le n$,
$w_n^*xw_n$ converge to $a$ in the weak operator topology.

Moreover we have convergence in the strong operator topology if and only if $c=0$ and in the $\ast$-strong operator topology if and only if $c=b=0$. 
\end{lemma}

\begin{proof}
We will show here only the second claim, that $w_n^*xw_n$ converges to $a$ strongly if and only if $c=0$. The third claim, which is the only one we use in this paper, follows directly from the second one. The first claim is proved similarly. 

Let $\epsilon>0$ and $\xi=\sum_{j=1}^\infty \xi_j e_j\in H$. 
Choose $N>0$ so that $\|\sum_{j>N} \xi_je_j\|<\frac{\epsilon}{8\|x\|}$, and write $\xi'=\sum_{j=1}^N \xi_je_j$ and $\xi''=\xi-\xi'$. 
Then for each $n\geq N$, we have

\begin{align*}\label{equal}
   \|w_n^*xw_n \xi-a\xi\|&=\|w_n^*x(\xi',0)-a\xi'+(w_n^*xw_n-a)\xi''|\numberthis\\
   &=\|w_n^*(a\xi', c\xi')-a\xi'+(w_n^*xw_n-a)\xi''\|\\
    &=\|w_n^*(0, c\xi')+w_n^*(a\xi',0)-a\xi'+(w_n^*xw_n-a)\xi''\|.
\end{align*}
Choose $M\geq N$ so that $\|p_{M}a\xi'-a\xi'\|<\epsilon/8$ where $p_{M}$ is the projection onto span$\{e_1,...,e_{M}\}$. Then we have $\|w_n^*(a\xi',0)-a\xi'\|<\epsilon/4$, $\|(w_n^*xw_n-a)\xi''\|<\epsilon/4$, and $\|w_n^*(0,c\xi')\|=\|c\xi'\|$ for all $n\geq M$. So (\ref{equal}) gives us 
$$\|c\xi'\|-\epsilon/2< \|w_n^*xw_n \xi-a\xi\|< \|c\xi'\|+\epsilon/2$$
for all $n\geq M$.
If $c=0$, then we have $\|w_n^*xw_n\xi-a\xi\|<\epsilon$. On the other hand, if there exists an $n>M$ such that $\|w_n^*xw_n \xi-a\xi\|<\epsilon/4$, then  
$$\|c\xi\|< \|c\xi'\|+\frac{\epsilon}{4}<\|w_n^*xw_n \xi-a\xi\|+\frac{3\epsilon}{4}<\epsilon.$$ 
Since $\epsilon$ and $\xi$ were arbitrary, we conclude that $w_n^*xw_n$ converges strongly to $a$ exactly when $c\xi=0$ for all $\xi\in H$. 
\end{proof}

We will need one more lemma, which is essentially the statement 5 in \cite[Lemma 1]{Had14}, where it is formulated in slightly different terms. For the reader's convenience we give a proof of it here.

\begin{lemma}[Hadwin \cite{Had14}]\label{UnitaryAppr} Let $u\in B(H)$ be a unitary operator and let $p_n\in B(H)$, $n\in \mathbb N$,  be a sequence of finite rank projections such that $p_n\uparrow 1$ in $\ast$-strong operator topology. Then for each $n\in \mathbb N$ there is a unitary operator $u_n\in B(p_nH)$ such that
$u_n \to u$ in $\ast$-strong operator topology, where we view $p_nH$ as a subspace of $H$ in the obvious way.
\end{lemma}
\begin{proof} Write $u$ as $u= e^{2\pi i a}$, where $-1\le a\le 1$, and let $u_n=e^{2\pi i p_nap_n}$. Since  $p_nap_n\to a$ in the $\ast$-strong topology and since the functional calculus is continuous with respect to the $\ast$-strong operator topology, one has $u_n \to u$ $\ast$-strongly.
\end{proof}

Our main theorem is for both unital and non-unital cases, meaning that C*-algebras can be either unital or non-unital, and in the case both of them are unital, the amalgamated free product can be either unital or non-unital. 

\begin{theorem}\label{main} Let $A$ and $B$ be separable strongly RFD $\mathrm{C}^*$-algebras and let $C$ be a central $\mathrm{C}^*$-subalgebra in both $A$ and $B$.
Then the amalgamated free product $A\ast_C B$ is RFD.
\end{theorem}
\begin{proof}
Let $0\neq x \in A\ast_C B$. We construct a finite dimensional representation $\sigma$ of $A\ast_C B$ such that $\sigma(x)\neq 0$. There exist an irreducible representation $\rho$ of $A\ast_C B$ on a separable Hilbert space $H$ and a unit vector $\xi\in H$ such that $\|\rho(x)\xi\|\ge \frac{\|x\|}{2}.$ Let $i_A$ and $i_B$ denote the standard embeddings of $A$ and $B$ into $A\ast_C B$.
Choose $K\in \mathbb{N}$, $a_i^{(k)}\in A, b_i^{(k)}\in B$, $i=1, \ldots, N^{(k)}$, $k = 1, \ldots, K$ such that for \begin{equation}\label{NotationalNightmare}\tilde x =  \sum_{k=1}^K i_A(a_1^{(k)})i_B(b_1^{(k)})\ldots i_A(a_{N^{(k)}}^{(k)})i_B(b_{N^{(k)}}^{(k)}),\end{equation}  we have
\begin{equation}\label{14}\|x - \tilde x \|\le \frac{\|x\|}{8}.\end{equation} 
(The sum in (\ref{NotationalNightmare}) might also contain monomials starting with an element from $i_B(B)$ or ending with an element from $i_A(A)$, but we will assume that it does not. This does not change anything in the proof and we do it just to avoid notational nightmare.)
Then 
\begin{equation}\label{13}\|\rho\left(\tilde x\right)\xi\|\ge \frac{3\|x\|}{8}.\end{equation}
We denote the representations of $A$ and $B$ induced by $\rho$ with 
$\rho_A$ and $\rho_B$ respectively, i.e., $\rho_A(a) = \rho(i_A(a)), \; \rho_B(b) = \rho(i_B(b)),$ for each $a\in A, b\in B$.
Let $$N = \max_{1\le k\le K}N^{(k)},$$ $$E = \{a_i^{(k)}\;|\; k = 1, \ldots, K, i= 1, \ldots, N^{(k)}\},$$ $$ F = \{b_i^{(k)}\;|\; k = 1, \ldots, K, i= 1, \ldots, N^{(k)}\},$$
and
\begin{align*}  G = \{\xi\}&\bigcup  \{ \rho_B(b_i^{(k)})\rho_A(a_{i+1}^{(k)})\rho_B(b_{i+1}^{(k)})\ldots \rho_A(a_{N^{(k)}}^{(k)})\rho_B(b_{N^{(k)}}^{(k)})\xi\;|\;i = 1, \ldots, N^{(k)},\\ & \hphantom{{}= \{ \rho_B(b_i^{(k)})\rho_A(a_{i+1}^{(k)})\rho_B(b_{i+1}^{(k)})\ldots \rho_A(a_{N^{(k)}}^{(k)})\rho_B(b_{N^{(k)}}^{(k)})\xi\;|\}} k=1, \ldots, K\}\\
& \bigcup  \{ \rho_A(a_i^{(k)})\rho_B(b_{i}^{(k)})\rho_A(a_{i+1}^{(k)})\ldots \rho_A(a_{N^{(k)}}^{(k)})\rho_B(b_{N^{(k)}}^{(k)})\xi\;|\;i = 2, \ldots, N^{(k)},\\ & \hphantom{{}\bigcup  \{ \rho_A(a_i^{(k)})\rho_B(b_{i}^{(k)})\rho_A(a_{i+1}^{(k)})\ldots \rho_A(a_{N^{(k)}}^{(k)})\rho_B(b_{N^{(k)}}^{(k)})\xi\;|\;}  k=1, \ldots, K\}.\end{align*}
 Since $\rho$ is irreducible and $C$ is a central subalgebra in $A\ast_C B$, 
for each $c\in C$ there is $\lambda(c) \in \mathbb C$ such that $$\rho(c) = \lambda(c) 1.$$
Since $A$ and $B$ are strongly RFD, $\rho_A(A)$ and $\rho_B(B)$ are RFD. Note that, unless $\rho(C)$ is zero, $\rho_A(A)$ and $\rho_B(B)$ are both unital, regardless of whether or not $A\ast_C B$ is.
  In this case let $\bar \pi_1, \bar \pi_2, \ldots$ be a countable separating family of unital finite dimensional representations of $\rho_A(A)$ and let $\pi_i = \bar \pi_i\circ \rho_A.$ Let $\bar \pi'_1, \bar \pi'_2, \ldots$ be a countable separating family of unital finite dimensional representations of $\rho_B(B)$ and let $\pi'_i = \bar \pi'_i\circ \rho_B.$  
  
  In the case $\rho(C)$ is zero, we replace unital representations by nondegenerate ones.
  
  Let $N_i = \dim \pi_i$, $N'_i = \dim \pi'_i$.  Let $$\tilde \pi_i = \pi_i^{(N_i')}, \; \tilde\pi'_i = {\pi'}_i^{(N_i)}.$$ Then $\dim \tilde \pi_i = \dim \tilde \pi'_i.$ Notice that for each $c\in C$, \begin{equation}\label{1} \tilde \pi_i(c) = \lambda(c) 1 = \tilde \pi'_i(c).\end{equation}
Let $\pi$ be a direct sum of all $\tilde \pi_i$'s where each one is repeated infinitely many times, say $$\pi = \tilde \pi_1 \oplus (\tilde \pi_1 \oplus \tilde \pi_2) \oplus (\tilde \pi_1 \oplus \tilde \pi_2 \oplus \tilde \pi_3) \oplus \ldots.$$ We consider $\pi$ as a representation of $A$ on $B(H)$ with respect to some decomposition $$H = H_1 \oplus H_2 \oplus H_3 \oplus \ldots$$ where $H_1 \cong \mathbb C^{N_1N_1'}, H_2 \cong \mathbb C^{N_1N_1' + N_2N_2'}, \ldots$.  Let $\pi': B \to B(H)$ be defined by $$\pi' = \tilde \pi'_1 \oplus (\tilde \pi'_1 \oplus \tilde \pi'_2) \oplus (\tilde \pi'_1 \oplus \tilde \pi'_2 \oplus \tilde \pi'_3) \oplus \ldots $$ with respect to the same decomposition of $H$. Let $p_i\in B(H)$ be the orthogonal projection onto $H_1\oplus \ldots \oplus H_i$.
It is easy to see that each $p_iH$ is an invariant subspace for $\pi$ and $\pi'$. 

 Now, for any $a\in A$, we have that $\pi(a) = 0$  iff $\rho_A(a) = 0$, and $rank (\pi(a)) = \infty$ when $\pi(a)\neq 0$. So by Voiculescu's theorem, $\rho_A \oplus \pi$ is approximately unitarily equivalent to $\pi$. Hence there exists a unitary $u: H\oplus H \to H$ such that for all $d \in E$ we have
 $$\left\|\left(\begin{array}{cc} \rho_A(d) & \\ & \pi(d) \end{array}\right) -  u^*\pi(d)u\right\|\le \delta,$$ where $$\delta = \frac{\|x\|}{80\cdot KN}.$$ By Lemma \ref{Don} there exist unitaries $w_m: H \to H\oplus H$ such that for all $a\in A$, $$w_m^* \left(\begin{array}{cc} \rho_A(a) & \\ & \pi(a) \end{array}\right) w_m \to \rho_A(a)$$ in the $*$-strong topology.
In particular there exists a unitary $w: H \to H\oplus H$ such that for all $d\in E$, $\eta \in G$
$$\left\|\rho_A(d)\eta - w^* \left(\begin{array}{cc} \rho_A(d) & \\ & \pi(d) \end{array}\right) w\eta\right\|<\delta.$$
Hence for all $d\in E, \eta \in G$, \begin{equation}\label{10}\|\rho_A(d)\eta-  w^*u^* \pi(d) uw \eta\|\le 2\delta.\end{equation}
Similarly we find unitaries $u', w'$ such that  \begin{equation}\label{10'}\|\rho_B(d)\eta-  w'^*u'^* \pi'(d) u'w' \eta\|\le 2\delta.\end{equation} for all $d\in F$, $\eta\in G$.
Applying Lemma \ref{UnitaryAppr} to $uw$ and $u'w'$, we find $M\in \mathbb N$ and unitaries $v$ and $v'$ on $p_mH$, which is identified with a subspace of $H$, such that
 $$\| \eta -  p_m\eta\|\le \delta,$$
 $$\|(uw-v)p_m\eta\| \le \delta,$$
$$\|(u'w'-v')p_m\eta\| \le \delta,$$ for all $\eta\in G$,
$$\| (w^*u^*-  v^*)\pi(d)uwp_m\eta\|\le \delta,$$
 for all $d\in E$, $\eta\in G$ and $$\| (w'^*u'^* -  v'^*)\pi'(d)u'w'p_m\eta\|\le \delta,$$
 for all $d\in F$, $\eta\in G$.

Now we define finite dimensional representations $\sigma_A: A \to p_mB(H)p_m$ and $\sigma_B: B \to p_mB(H)p_m$ by $$\sigma_A (a) =  v^* (p_m \pi(a) p_m) v, \;\;\; \sigma_B (b) =  v'^* (p_m \pi'(b) p_m) v',$$ for any $a\in A$, $b\in B$. Then  by (\ref{1}) $$\sigma_A(c) = \lambda(c) 1_{p_mH} = \sigma_B(c),$$ for any $c\in C$.  As $\sigma_A$ and $\sigma_B$ agree on $C$, we obtain a finite dimensional representation $\sigma: A\ast_C B \to p_mB(H)p_m$.

 Since $p_mH$ is invariant subspace for $\pi$, we have $p_m\pi(a)p_mv =\pi(a)v$, for each $a\in A$. Using this we obtain
 \begin{align*}\label{11} 
 & \|w^*u^*\pi(d)uw\eta - \sigma_A(d)p_m\eta\| \numberthis = \|w^*u^*\pi(d)uw\eta - v^* (p_m \pi(d) p_m) v p_m\eta\| \\
\le & \|w^*u^*\pi(d)uw(\eta - p_m\eta)\| + \|w^*u^*\pi(d)uwp_m\eta - v^*  \pi(d)  v p_m\eta\| \\ \le & \|w^*u^*\pi(d)uw(\eta - p_m\eta)\| + \|(w^*u^*-v^*)\pi(d)uwp_m\eta\| + \|v^*\pi(d)(uw-v)p_m\eta\| \\ \le & 3\delta,
 \end{align*}
 for any $d\in E$, $\eta\in G$.
By (\ref{10}) and (\ref{11}), we have
 \begin{equation}\label{20} \|\rho_A(d)\eta - \sigma_A(d)p_m\eta\| \le 5\delta, \end{equation} for any $d\in E$, $\eta\in G$. Similarly we obtain \begin{equation}\label{21} \|\rho_B(d)\eta - \sigma_B(d)p_m\eta\| \le 5\delta,
  \end{equation} for any $d\in F$, $\eta\in G$.   By repeated uses of (\ref{20}) and (\ref{21}),
\begin{align*}
& \| \sigma_A(a_1^{(k)})\sigma_B(b_1^{(k)})\ldots \sigma_A(a_{N^{(k)}}^{(k)})\sigma_B(b_{N^{(k)}}^{(k)})p_m\xi \\
&  \hphantom{{} \sigma_A(a_1^{(k)})\sigma_B(b_1^{(k)})} -
\rho_A(a_1^{(k)})\rho_B(b_1^{(k)})\ldots \rho_A(a_{N^{(k)}}^{(k)})\rho_B(b_{N^{(k)}}^{(k)})\xi\| \\
&  \le \|\sigma_A(a_1^{(k)})\sigma_B(b_1^{(k)})\ldots \sigma_A(a_{N^{(k)}}^{(k)})p_m\left(\sigma_B(b_{N^{(k)}}^{(k)})p_m\xi  - \rho_B(b_{N^{(k)}}^{(k)})\right)\xi \| \\
& \hphantom{{}\le}+ \| \sigma_A(a_1^{(k)})\sigma_B(b_1^{(k)})\ldots \sigma_A(a_{N^{(k)}}^{(k)})p_m\rho_B(b_{N^{(k)}}^{(k)})\xi \\
&\hphantom{{}  \hphantom{{}\le}+\sigma_A(a_1^{(k)})\sigma_B(b_1^{(k)})}  -\rho_A(a_1^{(k)})\rho_B(b_1^{(k)})\ldots \rho_A(a_{N^{(k)}}^{(k)})\rho_B(b_{N^{(k)}}^{(k)})\xi\| \\
&\le 5\delta  + \| \sigma_A(a_1^{(k)})\sigma_B(b_1^{(k)})\ldots \sigma_A(a_{N^{(k)}}^{(k)})p_m\rho_B(b_{N^{(k)}}^{(k)})\xi \\
& \hphantom{{}\le 5\delta  +\| \sigma_A(a_1^{(k)})\sigma_B(b_1^{(k)})}- \rho_A(a_1^{(k)})\rho_B(b_1^{(k)})\ldots \rho_A(a_{N^{(k)}}^{(k)})\rho_B(b_{N^{(k)}}^{(k)})\xi\| \\
& \leq 5\delta
\\
&\hphantom{{}\leq }+\|\sigma_A(a^{(k)}_1)\sigma_B(b^{(k)}_1)\ldots \sigma_B(b^{(k)}_{N^{(k)}-1})p_m\left(\sigma_A(a^{(k)}_{N^{(k)}})-\rho_A(a^{(k)}_{N^{(k)}})\right)(\rho_B(b^{(k)}_{N^{(k)}})\xi))\|\\
&\hphantom{{}\le }+\|\sigma_A(a^{(k)}_1)\sigma_B(b^{(k)}_1)\ldots \sigma_B(b^{(k)}_{N^{(k)}-1})p_m\rho_A(a^{(k)}_{N^{(k)}})\rho_B(b^{(k)}_{N^{(k)}})\xi \\
&  \hphantom{{} \hphantom{{}\le 5\delta}+\sigma_A(a_1^{(k)})\sigma_B(b_1^{(k)})} -\rho_A(a_1^{(k)})\rho_B(b_1^{(k)})\ldots \rho_A(a_{N^{(k)}}^{(k)})\rho_B(b_{N^{(k)}}^{(k)})\xi\| \\
&\le \ldots \le 2N^{(k)}5\delta \le 10 N\delta.\end{align*}
Hence
 \begin{equation}\label{12}\| \sigma(\tilde x) p_m\xi - \rho(\tilde x)\xi\|\le 10NK\delta = \frac{\|x\|}{8}.\end{equation}
Combining (\ref{12}), (\ref{13}) and (\ref{14})  we obtain
\begin{multline*}\|\sigma(x)p_m\xi\|\ge \|\sigma\left(\tilde x\right)p_m\xi\| -
\|\sigma\left(x - \tilde x\right)p_m\xi\| \\ \ge  \|\rho\left(\tilde x\right)\xi\|
 - \|\sigma(\tilde x)p_m\xi - \rho\left(\tilde x \right)\xi\| - \|\sigma\left(x- \tilde x\right)p_m\xi\| \ge \frac{\|x\|}{8} . \end{multline*}
Thus $\sigma(x)\neq 0$.
\end{proof}

\begin{remark} 

When the amalgam $C$ is trivial (that is $C =\mathbb C 1$ in the unital case and $C=0$ in the non-unital case), the assumption that the $\mathrm{C}^*$-algebras are strongly RFD can be omitted. One needs them only to be RFD because starting with any pair of separating families of unital (or nondegenerate, in the non-unital case) representations of $A$ and $B$ would lead to representations coinciding on $C$. Thus our proof recovers the Exel-Loring result that free products of separable RFD C*-algebras are RFD.
\end{remark}

\begin{remark}
 It follows from the proof that given an irreducible representation $\rho$ of $A\ast_C B$ on a separable Hilbert space $H$ and an arbitrary sequence $p_n \uparrow 1$ of projections on $H$, we actually can construct finite dimensional representations of $A\ast_C B$ living on subspaces of $p_nH$ and $\ast$-strongly converging to $\rho$.
\end{remark}

\begin{corollary} (Korchagin \cite{Kor14}) The amalgamated free product of commutative $\mathrm{C}^*$-algebras is RFD.
\end{corollary}

\begin{corollary}\label{discrete} Let $G_1$ and $G_2$ be virtually abelian discrete groups and let $\Lambda$ be a central subgroup in both $G_1$ and $G_2$. Then $\mathrm{C}^*(G_1*_\Lambda G_2)\simeq \mathrm{C}^*(G_1)*_{\mathrm{C}^*(\Lambda )}\mathrm{C}^*(G_2)$ is RFD.
\end{corollary}
\begin{proof} It was proved in \cite{Thoma68} that $\mathrm{C}^*$-algebras of discrete virtually abelian groups are subhomogeneous. Hence they are strongly RFD and Theorem \ref{main} applies. The isomorphism is well-known (see e.g. \cite[Lemma 3.1]{ESS}).
\end{proof}

We also have an application for amalgamated free products of locally compact groups. However, to our knowledge, the isomorphism from Corollary \ref{discrete} has not been addressed in the case of locally compact groups. Thus, before we can proceed, we must prove that this isomorphism holds, at least in the case where the common subgroup is open and central. 
The following useful fact was communicated to us by Ben Hayes.

\begin{proposition}\label{Ben} Suppose G is a locally compact Hausdorff group and $\Lambda\leq G$ is
an open subgroup. Then the restriction of any nondegenerate representation $\pi :
L^1(G) \to B(H)$ to the natural copy of $L
^1
(\Lambda)$ inside $L^1
(G)$ is also  nondegenerate.
\end{proposition}
\begin{proof} Since $\Lambda \subseteq G$ is open, we can naturally identify $L^1
(\Lambda)$ as a subalgebra of
$L^1
(G)$ by extending compactly supported functions on $\Lambda$ to be zero off their supports on $G$. Moreover, there exists an approximate unit of $L^
1
(G)$ that is contained
in $L^1
(\Lambda)$. Indeed, take a neighborhood basis of the unit of $G$ consisting of open
sets in $\Lambda$ with compact closure. Since $\Lambda$ is open in $G$, this forms a neighborhood
basis of the unit in $G$, and the normalized characteristic functions on these sets
give an approximate unit.
Now, under a nondegenerate representation  $\pi : L^1
(G) \to B(H)$, any
approximate unit of $L^1
(G)$ will converge strongly to the identity. Since our
particular approximate unit was contained in $L^1
(\Lambda)$, the same will hold when
$\pi$ is restricted to $L^1
(\Lambda)$. It follows that $\pi|_{L^1(\Lambda)}$
is also nondegenerate.
\end{proof}

\begin{proposition}
Let $G_1$ and $G_2$ be locally compact Hausdorff groups and $\Lambda \leq G_1, G_2$ an open central subgroup of both. Then $\Gamma=G_1*_\Lambda G_2$ is a locally compact Hausdorff group and
$$\mathrm{C}^*(\Gamma)\simeq \mathrm{C}^*(G_1)*_{\mathrm{C}^*(\Lambda)} \mathrm{C}^*(G_2).$$
\end{proposition}

\begin{proof}

By \cite[Theorem 5]{KM}, since $\Lambda$ is open and central, we know that $\Gamma$ is locally compact and Hausdorff.
Moreover, by the remark following Corollary 3 in \cite{KM}, since $\Lambda$ is open in $G_1$ and $G_2$, it follows that $G_1$ and $G_2$ are open in $\Gamma$. 
Choose the Haar measures on $G_1$ and $G_2$, inherited from the embeddings $\tilde{\iota}_i:G_i\to \Gamma$, $i=1,2$. As in Proposition \ref{Ben}, we have natural embeddings $L^1(\Lambda)\hookrightarrow L^1(G_i)$ for $i=1,2$, and similarly, we can embed $L^1(G_1), L^1(G_2)\hookrightarrow L^1(\Gamma)$ so that the embeddings agree on the respective copies of $L^1(\Lambda)$. The same arguments as in the discrete setting show that these embeddings extend to the full group $\mathrm{C}^*$-algebras, and moreover, the induced embeddings $\iota_i: \mathrm{C}^*(G_i)\hookrightarrow \mathrm{C}^*(\Gamma)$ agree on $\mathrm{C}^*(\Lambda)$ (see e.g. \cite[Proposition 8.8]{Pis03}). 

To verify that $\mathrm{C}^*(\Gamma)\simeq \mathrm{C}^*(G_1)*_{\mathrm{C}^*(\Lambda)} \mathrm{C}^*(G_2)$, it suffices to check that 
$\mathrm{C}^*(\Gamma)= \mathrm{C}^*(\iota_1(\mathrm{C}^*(G_1))\cup \iota_2(\mathrm{C}^*(G_2)))$ and that $\mathrm{C}^*(\Gamma)$ satisfies the desired universal property.

Let $B=\mathrm{C}^*(\iota_1(\mathrm{C}^*(G_1))\cup \iota_2(\mathrm{C}^*(G_2)))$. To show $B=\mathrm{C}^*(\Gamma)$, it suffices to show that for every nondegenerate representation $\pi:\mathrm{C}^*(\Gamma)\to B(H)$, 
$W^*(\pi(\mathrm{C}^*(\Gamma)))=W^*(\pi(B))$.
To that end, let $\pi:\mathrm{C}^*(\Gamma)\to B(H)$ be a nondegenerate representation, and let $\pi_i= \pi\circ\iota_i: \mathrm{C}^*(G_i) \to B(H)$, for $i=1,2$. Since $\iota_i(G_i)$ are both open in $\Gamma$, $\pi_1$ and $\pi_2$ are still nondegenerate by Proposition \ref{Ben}. Since there is
a 1-1 correspondence between nondegenerate representations of a full group $\mathrm{C}^*$-algebra of a locally compact group and strongly
continuous unitary representations of the group (see e.g. \cite[13.3.5]{Dix77} or \cite[p. 183-184]{Dav96}), $\pi, \pi_i$, $i=1,2$, correspond to unique strongly continuous unitary representations $\tilde{\pi}:\Gamma\to U(H)$ and $\tilde{\pi}_i:G_i\to U(H)$, $i=1,2$. Moreover, $W^*(\pi(\mathrm{C}^*(\Gamma)))=W^*(\tilde{\pi}(\Gamma))$ and $W^*(\pi_i(\mathrm{C}^*(G_i)))=W^*(\tilde{\pi_i}(G_i))$, $i=1,2$ (see again \cite[13.3.5]{Dix77} or \cite[p. 183-184]{Dav96}). Since $\tilde{\iota}_1(G_1)\cup \tilde{\iota}_2(G_2)$ generates $\Gamma$, 
we compute
\begin{align*}W^*(\pi(B)) &= W^*\left(\pi_1(\mathrm{C}^*(G_1))\cup \pi_2(\mathrm{C}^*(G_i))\right)\\
&= W^*\left(W^*(\pi_1(\mathrm{C}^*(G_1))), W^*(\pi_2(\mathrm{C}^*(G_2)))\right)\\
&= W^*\left(W^*(\tilde\pi_1(G_1)), W^*(\tilde\pi_2(G_2))\right)\\
&=W^*(\tilde\pi(\Gamma)) = W^*(\pi(\mathrm{C}^*(\Gamma))).
\end{align*}

Now, suppose $\phi_i:\mathrm{C}^*(G_i)\to B(H)$, $i=1,2$ are nondegenerate representations that agree on $\mathrm{C}^*(\Lambda)$. Again, these correspond to unique strongly continuous unitary representations $\tilde{\phi}_i:G_i\to U(H)$, $i=1,2$, which agree on $\Lambda$. The universal property of $\Gamma$ gives a unique strongly continuous unitary representation $\tilde{\psi}:\Gamma\to U(H)$ such that $\tilde{\psi}\tilde{\iota}_i=\tilde{\phi}_i$ for $i=1,2$. This induces a nondegenerate representation $\psi: \mathrm{C}^*(\Gamma)\to B(H)$. Moreover, for $i=1,2$ and $f\in L^1(G_i)$
\begin{align*}
    \psi\iota_i (f)&=\int_\Gamma \iota_i(f)(t)\tilde{\psi}(t) dt=\int_{\iota_i(G_i)} \iota_i(f)(t)\tilde{\psi}(t) dt\\
    &=\int_{G_1} \iota_i(f)(\iota_i(s))\tilde{\psi}(\tilde{\iota}_i(s)) d\iota_i(s)\\
    &=\int_{G_1} f(s)\tilde{\phi}_i(s) ds=\phi_i(f).
\end{align*}
Thus $\mathrm{C}^*(\Gamma)$ has the universal property of  $\mathrm{C}^*(G_1)*_{\mathrm{C}^*(\Lambda)}\mathrm{C}^*(G_2)$.
\end{proof}

\begin{corollary}\label{loccpt}
 Let $G_1$ and $G_2$ be separable locally compact groups, 
 and let $\Lambda$ be an open central subgroup in both. Assume moreover that $G_1$ and $G_2$ are Lie groups each containing a closed subgroup of finite index that is compact modulo its center or  are projective limits of such Lie groups. Then the full group $\mathrm{C}^*$-algebra of the amalgamated free product $G_1*_\Lambda G_2$ is RFD.
\end{corollary}
\begin{proof}
From \cite{KM}, we know $G_1*_\Lambda G_2$ exists and is a locally compact Hausdorff topological group (actually even a Lie group).
As was proved in \cite{Moo72}, the groups described are exactly the locally compact groups whose irreducible representations are all finite dimensional. Hence, their full group $\mathrm{C}^*$-algebras are FDI and, by assumption, separable. So, Theorem \ref{main} again applies.
\end{proof}
Moreover, the amalgamated groups from each of these corollaries are MAP (using results from \cite{SW} in the locally compact case).


But groups with FDI $\mathrm{C}^*$-algebras are not the only ones covered by Theorem \ref{main}. For example, in \cite{BGS}, the authors prove that there exist infinite discrete groups with just-infinite RFD full group $\mathrm{C}^*$-algebras. On the other hand, it follows from \cite{Thoma68} that a discrete group has FDI full group $\mathrm{C}^*$-algebra if and only if it is virtually abelian. Just as their $\mathrm{C}^*$-algebras form disjoint classes, so must these groups. This means, Corollaries \ref{discrete} and \ref{loccpt} can be phrased in more generality, but the generalized statements sound oblique without knowing which groups yield strongly RFD $\mathrm{C}^*$-algebras, which leads us to the following question(s).

\begin{question} Is there a nice characterization of all (discrete) groups that have strongly RFD $\mathrm{C}^*$-algebras?
\end{question}


The questions are particularly curious because the class of all such groups does not seem to fit neatly into any well-known classes of MAP or RF groups. Consider the discrete case. Clearly every quotient of such a group is MAP. Moreover, it follows from Rosenberg's theorem that all discrete groups with strongly RFD full group $\mathrm{C}^*$-algebras must be amenable since their reduced group $\mathrm{C}^*$-algebras must be RFD and hence quasidiagonal. With Bekka and Louvet's aforementioned result from \cite{BL00} in mind, one may make the naive guess that it is sufficient to be amenable with all quotients MAP, but this is wrong. 
Every quotient of a finitely generated nilpotent group is also finitely generated, and hence RF and amenable. However, if the full group $\mathrm{C}^*$-algebra of a finitely generated nilpotent group is strongly RFD, then the group must be virtually abelian. This is because the $\mathrm{C}^*$-algebra generated by any irreducible representation of a nilpotent group is simple (\cite{MR76}) and, as Thoma has shown (\cite{Thoma68}), having only finite dimensional irreducible representations is equivalent to being virtually abelian. 

\providecommand{\bysame}{\leavevmode\hbox to3em{\hrulefill}\thinspace}
\providecommand{\MR}{\relax\ifhmode\unskip\space\fi MR }
\providecommand{\MRhref}[2]{%
  \href{http://www.ams.org/mathscinet-getitem?mr=#1}{#2}
}
\providecommand{\href}[2]{#2}


\begin{thebibliography}{10}

\bibitem{Arc95}
R.~J. Archbold, \emph{On residually finite-dimensional {$\mathrm{C}^*$}-algebras}, Proc.
  Amer. Math. Soc. \textbf{123} (1995), no.~9, 2935--2937. \MR{1301006}

\bibitem{ADEL}
Scott Armstrong, Ken Dykema, Ruy Exel, and Hanfeng Li, \emph{On embeddings of
  full amalgamated free product {$\mathrm{C}^*$}-algebras}, Proc. Amer. Math. Soc.
  \textbf{132} (2004), no.~7, 2019--2030. \MR{2053974}

\bibitem{Bau63}
Gilbert Baumslag, \emph{On the residual finiteness of generalised free products
  of nilpotent groups}, Trans. Amer. Math. Soc. \textbf{106} (1963), 193--209.
  \MR{144949}

\bibitem{Bek99}
M.~B. Bekka, \emph{On the full {$\mathrm{C}^*$}-algebras of arithmetic groups and the
  congruence subgroup problem}, Forum Math. \textbf{11} (1999), no.~6,
  705--715. \MR{1725593}

\bibitem{BL00}
M.~B. Bekka and N.~Louvet, \emph{Some properties of {$\mathrm{C}^*$}-algebras associated
  to discrete linear groups}, {$\mathrm{C}^*$}-algebras ({M}\"{u}nster, 1999), Springer,
  Berlin, 2000, pp.~1--22. \MR{1796907}

\bibitem{BGS}
V.~Belyaev, R.~Grigorchuk, and P.~Shumyatsky, \emph{On just-infiniteness of
  locally finite groups and their {$\mathrm{C}^*$}-algebras}, Bull. Math. Sci.
  \textbf{7} (2017), no.~1, 167--175. \MR{3625854}

\bibitem{BD}
Nathanial~P. Brown and Kenneth~J. Dykema, \emph{Popa algebras in free group
  factors}, J. Reine Angew. Math. \textbf{573} (2004), 157--180. \MR{2084586}

\bibitem{Cho80}
Man~Duen Choi, \emph{The full {$\mathrm{C}^* $}-algebra of the free group on two
  generators}, Pacific J. Math. \textbf{87} (1980), no.~1, 41--48. \MR{590864}

\bibitem{Con76}
A.~Connes, \emph{Classification of injective factors. {C}ases {$II_{1},$}
  {$II_{\infty },$} {$III_{\lambda },$} {$\lambda \not=1$}}, Ann. of Math. (2)
  \textbf{104} (1976), no.~1, 73--115. \MR{0454659}

\bibitem{CS}
Kristin Courtney and Tatiana Shulman, \emph{Elements of {$\mathrm{C}^*$}-algebras
  attaining their norm in a finite-dimensional representation}, Canad. J. Math.
  \textbf{71} (2019), no.~1, 93--111. \MR{3928257}

\bibitem{Dav96}
Kenneth~R. Davidson, \emph{{$\mathrm{C}^*$}-algebras by example}, Fields Institute
  Monographs, vol.~6, American Mathematical Society, Providence, RI, 1996.
  \MR{1402012}

\bibitem{Dix77}
Jacques Dixmier, \emph{{$\mathrm{C}^*$}-algebras}, North-Holland Publishing Co.,
  Amsterdam-New York-Oxford, 1977, Translated from the French by Francis
  Jellett, North-Holland Mathematical Library, Vol. 15. \MR{0458185}

\bibitem{ESS}
S{\o}ren Eilers, Tatiana Shulman, and Adam~P.W. S{\o}rensen,
  \emph{{$\mathrm{C}^*$}-stability of discrete groups}, preprint, arXiv:1808.06793.

\bibitem{EL}
Ruy Exel and Terry~A. Loring, \emph{Finite-dimensional representations of free
  product {$\mathrm{C}^*$}-algebras}, Internat. J. Math. \textbf{3} (1992), no.~4,
  469--476. \MR{1168356}

\bibitem{FNT}
Tobias Fritz, Tim Netzer, and Andreas Thom, \emph{Can you compute the operator
  norm?}, Proc. Amer. Math. Soc. \textbf{142} (2014), no.~12, 4265--4276.
  \MR{3266994}

\bibitem{GMR}
Rostislav Grigorchuk, Magdalena Musat, and Mikael R\o~rdam, \emph{Just-infinite
  {$\mathrm{C}^*$}-algebras}, Comment. Math. Helv. \textbf{93} (2018), no.~1, 157--201.
  \MR{3777128}

\bibitem{Had77}
Don Hadwin, \emph{An operator-valued spectrum}, Indiana Univ. Math. J.
  \textbf{26} (1977), no.~2, 329--340. \MR{0428089}

\bibitem{Had14}
\bysame, \emph{A lifting characterization of {RFD} {$\rm \mathrm{C}^*$}-algebras}, Math.
  Scand. \textbf{115} (2014), no.~1, 85--95. \MR{3250050}

\bibitem{HS}
Don Hadwin and Tatiana Shulman, \emph{Stability of group relations under small
  {H}ilbert-{S}chmidt perturbations}, J. Funct. Anal. \textbf{275} (2018),
  no.~4, 761--792. \MR{3807776}

\bibitem{Hig51}
Graham Higman, \emph{A finitely related group with an isomorphic proper factor
  group}, J. London Math. Soc. \textbf{26} (1951), 59--61. \MR{0038347}

\bibitem{KM}
M.~S. Khan and Sidney~A. Morris, \emph{Free products of topological groups with
  central amalgamation. {II}}, Trans. Amer. Math. Soc. \textbf{273} (1982),
  no.~2, 417--432. \MR{667154}

\bibitem{Kir93}
Eberhard Kirchberg, \emph{On nonsemisplit extensions, tensor products and
  exactness of group {$\mathrm{C}^*$}-algebras}, Invent. Math. \textbf{112} (1993),
  no.~3, 449--489. \MR{1218321}

\bibitem{Kor14}
Anton Korchagin, \emph{Amalgamated free products of commutative
  {$\mathrm{C}^*$}-algebras are residually finite-dimensional}, J. Operator Theory
  \textbf{71} (2014), no.~2, 507--515. \MR{3214649}

\bibitem{LiShen}
Qihui Li and Junhao Shen, \emph{A note on unital full amalgamated free products
  of {RFD} {$\mathrm{C}^*$}-algebras}, Illinois J. Math. \textbf{56} (2012), no.~2,
  647--659. \MR{3161345}

\bibitem{LS04}
Alexander Lubotzky and Yehuda Shalom, \emph{Finite representations in the
  unitary dual and {R}amanujan groups}, Discrete geometric analysis, Contemp.
  Math., vol. 347, Amer. Math. Soc., Providence, RI, 2004, pp.~173--189.
  \MR{2077037}

\bibitem{Mal40}
A.~Malcev, \emph{On isomorphic matrix representations of infinite groups}, Rec.
  Math. [Mat. Sbornik] N.S. \textbf{8 (50)} (1940), 405--422. \MR{0003420}

\bibitem{Moo72}
Calvin~C. Moore, \emph{Groups with finite dimensional irreducible
  representations}, Trans. Amer. Math. Soc. \textbf{166} (1972), 401--410.
  \MR{302817}

\bibitem{MR76}
Calvin~C. Moore and Jonathan Rosenberg, \emph{Groups with {$T_{1}$} primitive
  ideal spaces}, J. Functional Analysis \textbf{22} (1976), no.~3, 204--224.
  \MR{0419675}

\bibitem{Pis03}
Gilles Pisier, \emph{Introduction to operator space theory}, London
  Mathematical Society Lecture Note Series, vol. 294, Cambridge University
  Press, Cambridge, 2003. \MR{2006539}

\bibitem{SW}
Nico Spronk and Peter Wood, \emph{Diagonal type conditions on group
  {$\mathrm{C}^*$}-algebras}, Proc. Amer. Math. Soc. \textbf{129} (2001), no.~2,
  609--616. \MR{1800241}

\bibitem{Thoma68}
Elmar Thoma, \emph{Eine {C}harakterisierung diskreter {G}ruppen vom {T}yp {I}},
  Invent. Math. \textbf{6} (1968), 190--196. \MR{0248288}

\end{thebibliography}
\end{document}